\documentclass[11pt]{amsart}
\usepackage{latexsym}
\usepackage{amssymb}
\usepackage{amsmath}
\usepackage{amsthm}
\usepackage{amsmath}
\usepackage{enumerate}
\usepackage{amsfonts}
\usepackage[mathscr]{eucal}
\usepackage{amssymb}
\usepackage[all]{xy}
\usepackage{mathrsfs}
\usepackage{setspace}
\usepackage{cancel} 
\usepackage{ulem}

\newtheorem{theorem}{Theorem}[section]
\newtheorem{lemma}[theorem]{Lemma}
\newtheorem{proposition}[theorem]{Proposition}

\newtheorem{defn}[theorem]{Definition}
\newtheorem{prop}[theorem]{Proposition}
\newtheorem{thm}[theorem]{Theorem}
\newtheorem{cor}[theorem]{Corollary}

\newtheorem{problem}[theorem]{Problem}

\theoremstyle{definition}
\newtheorem{remark}[theorem]{Remark}

\newcommand\tr{\mathop{\rm tr}}

\newcommand\osp{\mathop{\rm osp}}
\newcommand\osy{\mathop{\rm osy}}

\newcommand{\Span}{\operatorname{Span}}
\newcommand{\inner}[2]{\left\langle {#1}, {#2} \right\rangle}

\newcommand{\cl}[1]{\mathcal{#1}}
\newcommand{\bb}[1]{\mathbb{#1}}

\begin{document}

\title[Lov\'{a}sz theta type norms and Operator Systems]{Lov\'{a}sz theta type norms and Operator Systems}

\author[Carlos Ortiz and Vern Paulsen]{Carlos M. Ortiz and Vern I. Paulsen}

\address{Department of Mathematics, University of Houston,
Houston, Texas 77204-3476, U.S.A.}
\email{vern@math.uh.edu}

\address{Department of Mathematics, University of Houston,
Houston, Texas 77204-3476, U.S.A.}
\email{cortiz@math.uh.edu}


\begin{abstract} To each graph on $n$ vertices there is an associated subspace of the $n \times n$ matrices called the {\it operator system} of the graph. We prove that
two graphs are isomorphic if and only if their corresponding operator systems are unitally completely order isomorphic. This means that the study of graphs is equivalent to the study of these special operator systems up to the natural notion of isomorphism in their category. We define new graph theory parameters via this identification. Certain quotient norms that arise from studying the operator system of a graph give rise to a new family of parameters of a graph. We then show basic properties about these parameters and write down explicitly how to compute them via a semidefinte program, and discuss their similarities to the Lov\'{a}sz theta function. Finally, we explore a particular parameter in this family and establish a sandwich theorem that holds for some graphs.
\end{abstract}
\date{\today}


\maketitle
\section{Introduction}
The classic work of Shannon\cite{sh} associated a {\it confusability graph} to a binary channel and argued that the zero error capacity of the channel was a parameter definable solely in terms of this graph and its products. Later, Lov\'{a}sz\cite{lo} introduced his theta function, which he showed was an upper bound for Shannon's capacity. He presented many formulas for computing his theta function, which are optimization problems over a certain vector space of matrices associated with the graph. There is now a rich literature on Lov\'{a}sz's theta function and it plays an important role in both graph theory and binary information theory.
  
In analogy with the work of Shannon and Lov\'{a}sz, for a quantum channel, Duan, Severini and Winter \cite{dsw} have established that some notions of quantum capacity only depend on a vector space of matrices associated with the quantum channel, i.e., two quantum channels that define the same vector space have the same capacity. They argued that the study of these spaces of matrices should be treated as a kind of {\it non-commutative graph theory}.
In this paper we build upon that idea.

The vector spaces of matrices associated with a graph by Lov\'{a}sz and with a quantum channel by Duan, Severini and Winter are both examples of finite dimensional {\it operator systems}.  Given a graph $G$ we let $\cl S_G$ denote this operator system of matrices that is associated with $G$.

The natural notion of equivalence of operator systems is {\it unital, complete order isomorphism}. Our first main result shows that two graphs $G$ and $H$ are graph isomorphic if and only if the operator systems $\cl S_G$ and $\cl S_H$ are unitally, completely order isomorphic.  Thus, there is no difference between studying graphs and studying this special family of operator systems.  In particular, it should be possible to relate all graph parameters of $G$ to properties of $\cl S_G$. In this paper we are more interested in the converse. Namely, we begin with parameters that are ``natural'' to associate with operator systems and attempt to relate them to classical graph parameters.

The Lov\'{a}sz theta function naturally fits this viewpoint. Quotients of operator systems come equipped with two norm structures and we will show that a generalization of the theta function, introduced in \cite{dsw}, is an upper bound for the ratio between these two naturally ocuring norms.

\section{Preliminaries}

 As customary, we let $\cl B(\cl H)$ denote the space of bounded linear operators on some Hilbert space $\cl H$, let $M_n:= \cl B(\mathbb{C}^n)$, and let $E_{i,j}$ $1 \le i,j \le n$ be the canonical matrix units. We call a vector subspace $ \cl S \subseteq \cl B(\cl H)$ {\it *-closed} provided $X \in \cl S$ implies that $X^* \in \cl S,$ where $X^*$ denotes the adjoint of $X.$ We define $\cl S$ to be an {\it operator system} if $\cl S$ is a unital  $*$-closed subspace of $\cl B(\cl H)$. 

Operator systems are naturally endowed with a {\it matrix ordering} and can be axiomatically characterized in theses terms.  See, for example \cite{vpbook}. Briefly, given any vector space $\cl S,$ we let $M_n(\cl S)$ denote the vector space of $n \times n$ matrices with entries from $\cl S.$ We identify $M_n(\cl B(\cl H)) \equiv \cl B(\cl H \otimes \bb C^n)$ and let $M_n(\cl B(\cl H))^+$ denote the positive operators on the Hilbert space $\cl H \otimes C^n$. Given  $\cl S \subseteq \cl B(\cl H),$ we set $M_n(\cl S)^+ = M_n(\cl B(\cl H))^+ \cap M_n(\cl S).$   

The natural notion of equivalence between two operator systems is {\it unital, complete order isomorphism}.
Given two operator systems $\cl S$ and $\cl T$, a linear map $\phi:\cl S \to \cl T$ is called {\it completely positive} provided that for all $k,$ $(X_{i,j}) \in M_k(\cl S)^+$ implies that $(\phi(X_{i,j})) \in M_k(\cl T)^+.$ The map $\phi$ is {\it unital} when $\phi(I) = I.$ The map $\phi$ is called a {\it complete order isomorphism} if and only if $\phi$ is one-to-one, onto and $\phi$ and $\phi^{-1}$ are both completely positive. This last condition is equivalent to requiring that for all $n$,  $(X_{i,j}) \in M_n(\cl S)^+$ if and only if $(\phi(X_{i,j})) \in M_n(\cl T)^+.$

We will define a graph $G$ on $n$ vertices to be a subset of $\{1,2,...,n\}\times \{1,2,...,n\}$ with the property that $(i,j)\in G \iff (j,i)\in G$ for all $(i,j) \in G$ and $(i,i)\not\in G$ for all $i\in \{1,2,...,n\}$. We call the elements of $\{1,2,...,n\}$ the {\it  vertices} of G and say that two vertices $i$ and $j$ are {\it connected by an edge} when $(i,j) \in G.$ Given a graph $G$ on $n$ vertices we set $\widetilde{G}= G \cup \{ (i,i): 1 \le i \le n \}.$ We let $\overline{G}$ denote the {\it complement of the graph} $G$, that is, the graph with the property that $(i,j)\in \overline{G} \iff (i,j)\not \in \widetilde{G}.$ 

 Let 
\[R_G= \sum_{(i,j) \in \widetilde{G}} E_{i,j} =I+A_G\] where $I$ is the identity matrix and
\[A_G= \sum_{(i,j) \in G} E_{i,j}\]
 denotes the usual adjacency matrix of $G$. We define the {\it operator system of the graph $G$} to be $\cl S_G:=\Span\{E_{ij}: (i,j)\in \widetilde{G} \}.$

Given a self-adjoint $n \times n$ matrix $A$, we let $\lambda_1(A) \ge \ldots \ge \lambda_n(A)$ denote the eigenvalues of $A.$ It is known that $\lambda_1(A_G) \ge - \lambda_n(A_G)$, $\|A_G\|= \lambda_1(A_G)$, and  $\|R_G\|= 1 + \lambda_1(A_G)$ \cite{sp}. Let $G$ and $H$ be graphs on $n$ and $m$ vertices, respectively. We define $G\boxtimes H$ to be the {\it strong product} of the graphs, that is, the graph on $nm$ vertices with, 
 \begin{align*}
 &( (i,j), (k,l) )\in  G\boxtimes H \iff \\ 
 &(i,k)\in G\ and\ j=l\ or \\
  &(j,l)\in H\ and\ i=k\ or \\ 
 &(i,k)\in G\ and\ (j,l)\in H 
 \end{align*}
This product satisfies,
$$\cl S_{G\boxtimes H}=\cl S_G\otimes \cl S_H$$

\section{The Isomorphism Theorem}

In this section we prove that two graphs are isomorphic if and only if their operator systems are unitally, completely order isomorphic. This shows that the morphism $G \to \cl S_G$ in a certain sense loses no information. It suggests that there should be a dictionary for translating graph theoretical parameters into parameters of these special operator systems,  which one could then hope to generalize to all operator systems. In particular, the ``isomorphism'' problem for operator subsystems of $M_n$ is at least as hard as the isomorphism problem for graphs.

First, we do the ``easy'' equivalence.
Suppose that we are given two graphs $G_1, G_2$ on $n$
  vertices that are isomorphic via a permutation $\pi: \{1,...,n \} \to
\{1,...,n \},$ so that $G_2= \{ (\pi(i), \pi(j)): (i,j) \in G_1 \}.$
 If we define a linear map $U_\pi: \mathbb{C}^n \longrightarrow \mathbb{C}^n$ via $U_\pi (e_j) = e_{\pi(j)},$ where $\{ e_j: 1 \le j \le n \}$ denotes the canonical orthonormal basis for $\bb C^n,$
then it is not hard to see that $ U_{\pi}$ is a unitary matrix and that $U_\pi^* \cl S_{G_2} U_\pi = \cl S_{G_1}$. Moreover, the map $\phi: \cl B(\bb C^n)  \to \cl B(\bb C^n)$ defined by $\phi(X) = U_{\pi}^*X U_{\pi}$ is a unital, complete order isomorphism. Hence, the restriction $\phi: \cl S_{G_2} \to \cl S_{G_1}$ is a unital, complete order isomorphism between the operator systems of the graphs.

Conversely, if there exists a permutation such that $U_\pi^* \cl S_{G_2} U_\pi = \cl S_{G_1}$, then
$G_1$ and $G_2$ are isomorphic via $\pi$. To see this, note that we have
\begin{align*}
 (U_\pi E_{i,j} U_\pi^{-1})(e_k) &= U_\pi E_{i,j} e_{\pi^{-1}(k)}\\
                                                                                  &=U_\pi e_i  \hspace{0.3cm}  \textmd{(whenever} \hspace{0.3cm} j = \pi^{-1}(k)\hspace{0.3cm} \textmd{and}\hspace{0.3cm} 0 \hspace{0.3cm}\textmd{otherwise)}.\\
                                                                                  &=e_{\pi (i)} \hspace{0.3cm} \textmd{(since} \hspace{0.3cm} j = \pi^{-1}(k) \implies \pi (j) = k)
                                                                                  \end{align*}

                                                                                  Thus $U_\pi E_{i,j} U_\pi^{-1} = E_{\pi (i), \pi (j)}$.

The next result arrives at the same conclusion even when  the unitary is not induced by a permutation.

\begin{prop}
 Let $G_1$ and $G_2$ be graphs on $n$ vertices. If there exists a unitary $U$ such that $U^*\cl S_{G_1}U=\cl S_{G_2}$, then  $G_1$ and $G_2$ are isomorphic.
\end{prop}

\begin{proof}

Let $P_k = U^*E_{k,k}U$, $k = 1,\dots,n$
and $\cl C = {\rm span}\{P_k : k = 1,\dots,n\}$.
Since $\cl S_{G_1}$ is a bimodule over the algebra $\cl D_{n}$
of all diagonal matrices, $\cl S_{G_2}$ is a bimodule
over $\cl C$. Note that each $P_{k}$ is a rank one operator.

Write
$P_1 = (\lambda_i\overline{\lambda_j})_{i,j=1}^{n}$.
Set $\Lambda_1 = \{i : \lambda_i\neq 0\}$, and
renumber the vertices of $G_2$ so that
$\Lambda_1 = \{1,2,\dots,k\}$, for some $k \leq n$.
Suppose that $E_{i,j}\in  \cl S_{G_1}$ for some $i\in \{1,\dots,k\}$ and some $j > k$.
We have that the matrix $P_1 E_{i,j}$ has as its $(l,j)$-entry, where
$l\in \{1,\dots,k\}$, the scalar $\lambda_l\overline{\lambda_j}\neq 0$.
It follows that if $(i,j)\in G_2$, where
$i\in \{1,\dots,k\}$ and $j > k$, then $(l,j)\in G_2$ for all $l = 1,\dots,k$.

It now follows that if $W_1\in M_n$ is a unitary matrix of the form
$W_1 = V\oplus I_{n-k}$, where $V\in M_k$ is unitary and
$I_{n-k}$ is the identity of rank $n-k$, then $W_1^* \cl S_{G_2}W_1 = \cl S_{G_2}$.
Choose such a $W_1$ with the property that $V^*P_1V = E_{1,1}$.
Then $W_1^*U^* \cl S_{G_1}UW_1 = \cl S_{G_2}$ and
$W_1^*U^*E_{1,1}UW_1 = E_{1,1}$.

Now let $Q_2 = W^*U^*E_{2,2}UW$; then $Q_2$ is a rank one operator in $\cl S_{G_2}$;
write
$Q_2 = (\mu_i\overline{\mu_j})_{i,j=1}^n$ and set
$\Lambda_2 = \{i : \mu_i\neq 0\}$. Since $E_{1,1}E_{2,2} = E_{2,2}E_{1,1} = 0$,
we have that $E_{1,1} Q_2 = Q_2E_{1,1} = 0$.
This implies that $1\not\in \Lambda_2$. Now proceed as in the previous paragraph to
define a unitary $W_2\in M_n$ such that $W_2^*W_1^*U^* \cl S_{G_1}UW_1W_2 = \cl S_{G_2}$
and, after a relabeling of the vertices of $G_2$, we have that
$W_2^*W_1^*U^*E_{1,1}UW_1W_2 = E_{1,1}$ and $W_2^*W_1^*U^*E_{2,2}UW_1W_2 = E_{2,2}$.

A repeated use of the above argument shows that, up to a relabeling of the vertices of $G_2$,
we may assume that there exists a unitary $W\in M_n$ such that
$W^*\cl S_{G_1}W = \cl S_{G_2}$ and $W^*E_{i,i}W = E_{i,i}$ for each $i$.
But this means that $We_i = \lambda_i e_i$ with $|\lambda_i| =1$ for each $i$ (here $\{e_i\}$ is the standard basis of $\bb{C}^n$).
Hence $W$ is a diagonal unitary, and so $W^*\cl S_{G_1}W =\cl S_{G_1}$ and so
up to re-ordering, $\cl S_{G_1} = \cl S_{G_2}$, which implies that
$G_1$ is isomorphic to $G_2$.

\end{proof}

Given any operator system $\cl S$, each time we choose a unital complete order embedding  $\gamma:\cl S \to\cl B(\cl H)$ we can consider the C*-algebra generated by the image, $C^*(\gamma(S)) \subseteq \cl B(\cl H).$ The theory of the C*-envelope guarantees that among all such generated C*-algebras, there is a universal quotient, denoted $C^*_e(\cl S)$ and called the {\it C*-envelope} of $\cl S.$  See \cite[Chapter~a]{vpbook}.

\begin{thm} Let $G$ be a graph on $n$ vertices.  Then the C*-subalgebra of $M_n$ generated by $\cl S_G$ is the C*-envelope of $\cl S_G.$
\end{thm}
\begin{proof} Let $C^*(\cl S_G) \subseteq M_n$ be the C*-subalgebra generated by $\cl S_G.$  By the general theory of the C*-envelope, there is a *-homomorphism $\pi: C^*(\cl S_G) \to C^*_e(\cl S_G)$ that is a complete order isomorphism when restricted to $\cl S_G.$

First assume that $G$ is connected.
Then for any $i$ and $j$ if one uses a path from $i$ to $j$ in $G_1$ then this path gives a way to express $E_{i,j}$ as a product of matrix units that belong to $\cl S_{G}.$ Thus, the C*-subalgebra of $M_n$ generated by $\cl S_{G}$ is all of $M_n.$ But since $M_n$ is irreducible, $\pi$ must be an isomorphism.

For the general case, assume that $G$ has connected components of sizes $n_1, ..., n_k$ with $n_1+ \cdots + n_k = n.$ By the argument above one can see that $C^*(\cl S_G) \equiv M_{n_1} \oplus \cdots \oplus M_{n_k}.$ If for each component $C_j$ one lets $P_j= \sum_{i \in C_j} E_{i,i},$ then these projections belong to the center of $C^*(\cl S_G)$ and $P_j C^*(\cl S_G) P_j$ is *-isomorphic to $M_{n_j}.$
Also, their images $\pi(P_j)$ belong to the center of $C^*_e(\cl S_G).$ 

Thus, $\pi(P_j) C^*_e(\cl S_G) \pi(P_j)$ is either 0 or *-isomorphic to $M_{n_j}.$

Look at the diagonal matrices $\cl D_n \subseteq \cl S_G.$ Since $\pi$ is a *-homomorphism on the subalgebra and a complete order isomorphism on this subalgebra, it is a *-isomorphism when restricted to $\cl D_n.$ Thus, $\pi(P_j) \ne 0$  and so,
these central projections allow us to decompose $C^*_e(\cl S_G) = A_1 \oplus \cdots A_k,$ with $A_j \equiv M_{n_j}.$
\end{proof}

\begin{thm} Let $G_1$ and $G_2$ be graphs on $n$ vertices. The following are equivalent:
\begin{enumerate}
\item
$G_1$ is isomorphic to $G_2$,
\item there exists a unitary $U$ such that $U^*\cl S_{G_1} U =\cl S_{G_2},$
\item  $\cl S_{G_1}$ is unitally, completely order isomorphic to $\cl S_{G_2}$.
\end{enumerate}
\end{thm}
\begin{proof}
We have shown above that (1) implies (3) and that (2) implies (1).  It remains to prove that (3) implies (2).

So assume that (3) holds and let $\phi:\cl S_{G_1} \to\cl S_{G_2}$ be a unital, complete order isomorphism. In this case, by \cite[Theorem~a.b]{vpbook} $\phi$ extends uniquely to a *-isomorphism, which we will denote by $\rho$, between their C*-envelopes. Since, by the previous theorem, the C*-envelopes are just the C*-subalgebras that they generate, we have $\rho: C^*(\cl S_{G_1}) \to C^*(\cl S_{G_2})$ is a unital *-isomorphism.

Suppose first that $G_1$ is connected. Then $M_n= C^*(\cl S_{G_1})$ is all of $M_n.$ Thus, $dim(C^*(\cl S_{G_2}))= dim(C^*(\cl S_{G_1})) = n^2,$ which forces $C^*(\cl S_{G_2}) = M_n.$

 Hence, $\rho:M_n \to M_n$ is a *-isomorphism. But every *-isomorphism of $M_n$ is induced by conjugation by a unitary, and so (2) holds.

Now assume that $G_1$ has connected components of sizes $n_1,...,n_k,$ with $n_1+ \cdots + n_k = n.$  In this case, applying the last theorem, we see that $C^*(\cl S_{G_1}) \equiv M_{n_1} \oplus \cdots \oplus M_{n_k}\equiv C^*_e(\cl S_{G_1}) \equiv C^*_e(\cl S_{G_2}) \equiv C^*(\cl S_{G_2}).$ Since $C^*(\cl S_{G_2}) \equiv M_{n_1} \oplus \cdots \oplus M_{n_k}$ one sees that $G_2$ has components of sizes $n_1, ..., n_k$ as well.

The central projections onto these components decomposes $\bb C^n$ into a direct sum of subspaces of dimensions $n_1, ..., n_k$ in two different ways and on each subspace the complete order isomorphism is implemented by conjugation by a unitary.  Thus, the complete order isomorphism is implemented by conjugation by the direct sum of these unitaries.
\end{proof}

\section{Quotients of Operator Systems and the Lov\'{a}sz Theta Function}
In this section we introduce some natural operator system parameters, which when specialized to graphs we will see are related to Lov\'{a}sz's theta function.

Given an operator system $\cl S,$ a subspace
 $\cl J\subseteq \cl S$ is called a {\it kernel} if there is an operator system $\cl T$ and a unital, completely positive (UCP) map $\phi: \cl S \to \cl T$ such that $\cl J = ker(\phi).$ Since every operator system $\cl T$ unital complete order embedding into $\cl B(\cl H)$ for some $\cl H$. There is no lost in generality in assuming that $\cl T = \cl B(\cl H)$ in the definition of a kernel.

  In \cite{kptt2010}, it was shown that the vector space quotient $\cl S/\cl J$ can be turned into an operator system, called the {\it quotient operator system} as follows.  Let $\cl D_n(\cl S/\cl J)$ be the set of all $ (x_{i,j}+ \cl J)\in M_n(\cl S/\cl J)$ for which there exists 
$(y_{i,j}) \in M_n(\cl J)$ such that $(x_{i,j} + y_{i,j})\in M_n(\cl S)^+$. 
Let $M_n(\cl S/\cl J)^+$ be the {\it Archimedeanisation} of $\cl D_n(\cl S/\cl J)$;  that is
$ (x_{i,j} + \cl J)\in M_n(\cl S/\cl J)^+$ if and and only if for every $\epsilon > 0$, 
$(x_{i,j} + \cl J) + \epsilon \dot 1_n\in \cl D_n(\cl S/\cl J)$. Here, $1_n$ is the element of 
$M_n(\cl S)$ whose  diagonal entries are all equal to  $1$ and all other entries are zero. Also, if $\cl J$ is finite dimensional, then we know that $\cl D_n(\cl S/\cl J)=M_n(\cl S/\cl J)^+$ so that this Archimedeanisation process is unnecessary by \cite{athesis}.

Every operator system is also an operator space. For this reason,
the quotient $\cl S/\cl J$ carries two,
in general distinct, operator space structures.
One is the canonical quotient operator space structure on $\cl S/\cl J$ 
arising from the fact that $\cl S$ and $\cl J$ are operator spaces.
On the other hand, the operator system quotient $\cl S/\cl J$ is an operator system and so carries a norm. Examples have been given to show that these two norms can be quite different. See  \cite{fp} for some important examples of this phenomenon.

To simplify notation, given $x \in \cl S$ we shall set $\dot x := x + \cl J \in \cl S/\cl J,$ and for $X= (x_{i,j}) \in M_n(\cl S)$ we set $\dot X := ( x_{i,j} + \cl J) \in M_n(\cl S/\cl J).$

Following \cite{kptt2010}, given $X \in M_n(\cl S)$ so that $\dot X \in M_n(\cl S/\cl J)$ we  let $\|\dot X\|_{\osp}$
(resp. $\|\dot X\|_{\osy}$) denote the operator space 
(resp. the operator system) quotient norm.  It is known that $\|\dot X \|_{\osy} \le \|\dot X\|_{\osp}$ for every $X \in M_n(\cl S)$ and every $n.$

We identify a kernel $\cl J$ in the operator system $\cl S$ with a kernel $\cl K$ in the operator system $\cl T$ provided the operator systems $\bb C 1 + \cl J$ and $\bb C 1 + \cl K$ are unitally completely order isomorphic.

\begin{defn} Let $\cl S$ be an operator system and let $\cl J
  \subseteq \cl S$ be a kernel. Then the {\bf relative n-distortion} is
\[ \delta_n(\cl S, \cl J) = \sup \{ \frac{\|\dot X \|_{\osp}}{\|\dot X
  \|_{\osy}}: X \in M_n(\cl S) \}\]  and we call $ \delta_{cb}(\cl S,
\cl J) = \sup \{ \delta_n(\cl S, \cl J) : n \in \bb N \}$ the {\bf
relative complete distortion}.
We call 
\[ \delta_n(\cl J) = \sup \{ \delta_n(\cl S, \cl J) \}\]
the {\bf absolute n-distortion}  and $\delta_{cb}(\cl J) = \sup \{
\delta_n(\cl J): n \in \bb N \}$ the {\bf complete distortion},
where the supremum is taken over all operator systems $\cl S$ that contain $\cl J$ as a kernel.
\end{defn}

When $n=1$ we simplify the notation by setting $\delta(\cl S, \cl J) = \delta_1(\cl S, \cl J)$ and $\delta(\cl J) = \delta_1(\cl J).$  We now wish to relate this to a Lov\'{a}sz theta type parameter, which was first introduced in \cite{dsw}.

\begin{defn} Let $\cl S$ be an operator system and let $\cl J \subseteq \cl S$ be a kernel. Then we set
\begin{multline*} \vartheta_n(\cl J) = \sup \{ \| 1_n +J \|_{M_n(\cl S)}: J \in M_n(\cl J) \text{ and } 1_n + J \ge 0 \} \\ \text{ and } \vartheta_{cb}(\cl J) = \sup \{ \vartheta_n(\cl J): n \in \bb N \}.\end{multline*}
\end{defn}

Again when $n=1$ we set $\vartheta(\cl J) := \vartheta_1(\cl J).$

\begin{remark}\label{tr0} If we let $\cl S= M_n$ and let $\cl J$ denote the set
  of diagonal matrices of trace 0, then $\cl J$ is a kernel and it
  follows from the characterization of the quotient $M_n/\cl J$ in
  \cite{fp} that $n \le \delta(M_n, \cl J).$ For any $J \in \cl J$ we
  see that $tr(I_n+J) =n$ and so when $I_n+ J \ge 0$ we see that
  $\|I_n+J\| \le n$. Letting $J$ be the
  diagonal matrix with diagonal entries, $(n-1, -1, \ldots, -1)$ we
  see have $\|I_n+J\|=n,$ and so $\vartheta(\cl J) = n.$
\end{remark}


\begin{theorem}\label{th_te}
We have that $\delta(\cl J)\leq \vartheta(\cl J)$ and $\delta_{cb}(\cl J)\leq \vartheta_{cb}(\cl J)$. 
\end{theorem}
\begin{proof}
Let $x\in \cl S$ be such that $ \|\dot x\|_{\osy} = 1$. 
Then
$$\left(
\begin{matrix} 
  \dot 1_{\cl S} & \dot x\\ 
\dot x^* &    \dot 1_{\cl S}\\
\end{matrix}
\right)\in M_2(\cl S/\cl J)^+.$$
Thus,  for every $\epsilon >0,$
$$\left(
\begin{matrix} 
(1+\epsilon)\dot 1_{\cl S} & \dot x\\ 
\dot x^* &  (1+\epsilon)\dot 1_{\cl S}\\
\end{matrix}
\right)\in D_2(\cl S/\cl J)$$
and so there exists 
$\left(\smallmatrix
a & b\\ 
b^* & c
\endsmallmatrix
\right) \in M_2(\cl J)$ such that 
$$\left(
\begin{matrix} 
(1+\epsilon)1_{\cl S} + a &  x + b\\ 
x^* + b^* &  (1+\epsilon)1_{\cl S} + c\\
\end{matrix}
\right)\in M_2(\cl S)^+.$$
But then 
$$\|x + b\| \leq \max\{\|(1+ \epsilon)1_{\cl S} + a\|, \|(1+ \epsilon)1_{\cl S} + c\|\}$$
with $(1+\epsilon)1_{\cl S} + a, (1+ \epsilon)1_{\cl S} + c\in \cl S^+$.
Since $\epsilon$ was arbitrary, we have that $\|x+b\| \le \vartheta(\cl J)$
On the other hand,
$$\|x + b\| \geq \inf\{\|x + y\| : y\in \cl J\} = \|\dot x\|_{\osp}$$
and it follows that $\vartheta(\cl J) \geq \|\dot x\|_{\osp}.$
Thus,  $\delta(\cl S, \cl J) \le \vartheta(\cl J)$ for every $\cl S$ and so $\delta(\cl J) \le \vartheta(\cl J).$

Note that $M_n(\cl J)$ is a kernel in $M_n(\cl S)$ and $\delta_n(\cl S, \cl J) = \delta_1(M_n(\cl S), M_n(\cl J)).$ Also, $\vartheta(M_n(\cl J))
= \vartheta_n(\cl J).$ Hence, 
\[\delta_{cb}(\cl J) = \sup_n \{ \delta(M_n(\cl J)) \} \le \sup_n \{ \vartheta(M_n(\cl J)) \} = \vartheta_{cb}(\cl J).\] 
\end{proof}

\begin{cor}
For any $X\in M_n(\cl S)$,
$$||\dot X||_{osp}\leq \vartheta_n(\cl J) \cdot ||\dot X||_{osy}$$
\end{cor}

We now compute these parameters in one case.

\begin{cor} If $\cl J \subseteq M_n$ denotes the diagonal matrices of
  trace 0, then
\[ n = \delta(M_n, \cl J) = \delta(\cl J) = \vartheta(\cl J)=
\vartheta_{cb}(\cl J).\]
\end{cor}
\begin{proof}
By Remark~\ref{tr0} and the above result, we have that
\[n \le \delta(M_n, \cl J) \le \delta(\cl J) \le \vartheta(\cl J)
=n.\]
So all that remains is to show that $\vartheta_{cb}(\cl J) = n.$  

If
we let $\cl D_n\subseteq M_n$ denote the diagonal matrices, then for
each p, $M_p(\cl D_n)$ can be thought of as the C*-algebra of
functions from the set $\{1,...,n \}$ into $M_p.$ From this it can be seen
that every $(J_{k,l})
\in M_p(\cl D_n)$ is unitarily via an element in this algebra to a
diagonal element $diag(J_1,...,J_p)$ of this algebra.  Moreover, since
each $J_i$ is a linear combination of the matrices $J_{k,l}$ it
follows that if  $tr(J_{k,l})=0$ for all $k,l,$ then $tr(J_i) =0$ for
all $i.$  Since unitaries preserve norms, we see that
if $J_{k,l} \in \cl J$ and $diag(I_n,..., I_n) +(J+{k,l}) \ge 0,$ then
$I_n +J_i \ge 0.$ Also, $\| diag(I_n,...,I_n) +(J_{k,l})\| = \max \{
\|I_n +J_1\|,..., \|I_n +J_p\|\} \le \vartheta(\cl J).$

This shows that $\vartheta_{cb}(\cl J) = \vartheta(\cl J)$ and the
result follows.
\end{proof}

Note  that $M_k$ is a Hilbert space with respect to the inner product 
$(a,b) = \tr(ab^*)$, $a,b\in M_k$. Thus, given any subspace $\cl S\subseteq M_n$, 
one may form the orthogonal complement $\cl S^{\perp}$ of $\cl S$. 
Given a graph $G$ on $k$ vertices, 
$$\cl S_G^{\perp} = {\rm span}\{E_{i,j} : (i,j)\in G^c\}.$$
Results in \cite{dsw} imply that $\cl S_G^{\perp}$ is a kernel in our sense. Below is a direct proof in the language of operator systems,
that also characterizes the quotient as the operator system dual of $\cl S_G.$

We recall that given a finite dimensional operator system, $\cl S$, the dual space $\cl S^d$ is also an operator system. The matrix
ordering on the dual space is defined by $(f_{i,j}) \in M_n(\cl S^d)^+$ if and only if the map $F: \cl S \to M_n$ given by $F(x) = (f_{i,j}(x))$ is completely
positive.

\begin{proposition}\label{p_gker}
Let $G$ be a graph on $k$ vertices. Then $\cl S_G^{\perp}$ is a kernel in $M_k$ and the quotient $M_k/\cl S_G^{\perp}$ is completely
order isomorphic to the operator system dual $\cl S_G^d.$ 
\end{proposition}
\begin{proof} It is proven in \cite[Thm. 6.2]{ptt} that $M_k$ is self-dual as an operator system via the map $\rho: M_k \to M_k^d$ that sends the matrix unit $E_{i,j} \in M_k$ to
the dual functional $\delta_{i,j} \in M_k^d.$
Let $\iota : \cl S_G\to M_k$ be the inclusion map; it is clearly a complete order embedding. 
Thus its dual 
$\iota^d : M_k^d \to \cl S_G^d$ is a complete quotient map by \cite[Prop. 1.8]{fp}. Let $\cl J$ be its kernel. 
A functional $f = \sum_{i,j}\lambda_{i,j}\delta_{i,j}$ is in the kernel of $\iota^d$
if and only if 
$ f(E_{i,j}) = 0$ whenever $(i,j)\in G$ or $i = j$. Thus, $f$ is in the kernel of $\iota^d$ if and only if
$\lambda_{i,j} = 0$ whenever $(i,j)\in G$ or $i = j$. Thus,
$$\ker \iota^d = {\rm span}\{\delta_{i,j} : (i,j)\in G^c\}.$$
Thus, 
$$\rho^{-1}(\ker\iota^d) = {\rm span}\{E_{i,j} : (i,j)\in G^c\} = \cl S_G^{\perp}.$$
It follows that $\cl S_G^d \equiv M_k^d/\ker \iota^d \equiv M_k/\cl S_G^{\perp}.$
\end{proof}

\begin{cor} Let $G$ be a graph on $k$ vertices, let $x= \sum_{i,j=1}^k x_{i,j} E_{i,j} \in M_k$ and let $f= \sum_{i,j=1}^k x_{i,j} \delta_{i,j}: \cl S_G \to \bb C$ denote the corresponding functional.  Then 
\[ \|f\|= \| \dot x \|_{\osy} \ge \delta(M_k, \cl S_G^{\perp})^{-1} \| \dot x \|_{\osp} \ge \vartheta(\cl S_G^{\perp})^{-1} \|\dot x\|_{\osp}.\]
\end{cor}

We now see that 
\[\vartheta(\cl S_G^{\perp}) = \sup \{ \| I+K \|: I+K \ge 0, K \in \cl S_G^{\perp} \} = \vartheta(G)\]
by  \cite{dsw}. Similarly, 
$\vartheta_{cb}(\cl S_G^{\perp})= \tilde{\vartheta}(\cl S_G)$ is the \lq\lq complete'' Lov\'{a}sz number of $G$ 
introduced in \cite{dsw}.

In \cite{dsw} it is shown that for graphs, 
\[ \vartheta_{cb}(\cl S_G^{\perp}) = \vartheta(\cl S_G^{\perp}).\]
It is useful to recall their argument.

First, note that $M_p(\cl S_G) = \cl S_{G \boxtimes K_p}$, where $K_p$ denotes the complete graph on $p$ vertices.
Also notice that 
$$\cl S_{G \boxtimes K_p}^{\perp} = M_p(\cl S_G^{\perp})$$
Hence,
\[\vartheta_{cb}(\cl S_G^{\perp})= \sup_p \vartheta(\cl S_{G \boxtimes K_p}^{\perp}) = \sup_p \vartheta(G \boxtimes K_p) = \sup_p \vartheta(G)\vartheta(K_p) = \vartheta(G),\]
using Lov\'{a}sz famous result that $\vartheta$ is multiplicative for strong products of graphs
and the fact that $\vartheta(K_p) =1.$

We now get a lower bound on the distortion in terms of a graph
theoretic parameter. 

\begin{theorem}
Let $G$ be a graph on $k$ vertices and let $K_{p,q}$ be an induced complete bipartite subgraph of $G$. Then 
\[ \sqrt{pq} \le \delta(M_k, \cl S_G^{\perp}).\]
\end{theorem}
\begin{proof} Let the vertices for the subgraph be numbered $1,...,p$ for the first set and $p+1,...,p+q$ for the remainder. Let $X=(x_{i,j})$ be the matrix with
$x_{i,j} =1$ for $1\le i \le p$ and $p+1 \le j \le p+q$ and 0 otherwise.
Let $K=(k_{i,j})$ with $k_{i,j} =1$ for $1 \le i,j \le p$ and $i \ne j$ and 0 otherwise.  Let $R= (r_{i,j})$ be the matrix such that $r_{i,j} = 1$ for $p+1 \le i,j \le p+q$, $i \ne j$ and 0 otherwise.  Then $K,R \in \cl S_G^{\perp}$ and
\[ \begin{pmatrix} I+K & X\\X^* & I+R \end{pmatrix} \]
is positive.  Hence,  $\|X\|_{osy} \le 1.$  However,
\[ \|\dot X\|_{osp} = dist(X, \cl S_G^{\perp})= \|X\| = \sqrt{pq}.\]

Hence,  $\frac{||X+\cl \cl S_G^{\perp}||_{osp}}{||X+\cl S_G^{\perp}||_{osy}}\ge \sqrt{pq}.$\\
\end{proof}

\begin{remark} Haemers \cite{haemers} introduces the parameter
$\Phi(G) = \max \{ \sqrt{pq}: K_{p,q} \subseteq G \},$ i.e., the
maximum over all complete bipartite subgraphs of $G$, that are not
necessarily induced subgraphs. He proves that $\Phi(G) \le
\vartheta^{\prime}(G),$ which is
another variant of the Lovasz theta function. We have been unable to find any relationship
between his parameters and ours.
\end{remark}

If we let $\cl S = M_n$ and let $\cl T = \{ \begin{pmatrix} A & 0 \\ 0
  & A \end{pmatrix} : A \in M_n \}$ then these operator systems are
unitally, completely order isomorphic, but $\vartheta(\cl S^{\perp})
=1,$ while $\vartheta(\cl T^{\perp}) = 2.$
However, $\delta(M_n, \cl S^{\perp}) = \delta(M_{2n}, \cl T^{\perp})
=1.$
This motivates the following problems.

\begin{problem} If $\cl S \subseteq M_n$ and $\cl T \subseteq M_n$ are
  unitally completely order isomorphic, then is $\vartheta(\cl
  S^{\perp}) = \vartheta(\cl T^{\perp})$ ?
\end{problem}
\begin{problem} If $\cl S \subseteq M_n$ and $\cl T \subseteq M_m$ are
  unitally completely order isomorphic, then is $\delta(M_n, \cl
  S^{\perp}) = \delta(M_m, \cl T^{\perp})$ ?
\end{problem}
\begin{problem} Is $\delta_{cb}( \cl J) = \vartheta_{cb}( \cl J)$ ?
\end{problem}

\section{Multiplicativity of Graph Parameters}
One of the great strengths of the Lov\'{a}sz theta function is the
fact that it is multiplicative for strong graph product.  Recall that,
$$\vartheta(G)=\vartheta(\cl S_G^{\perp}) = \sup\{\|I + K\| : K\in
\cl S_G^{\perp}, I+K\in M_n^+\}.$$ 

In this section we wish to examine multiplicativity of some of the
other parameters. We have been unable to determine if our general
theta function is multiplicative for tensor products of kernels or if any of the various
distortions are multiplicative.

Instead we focus more closely on the graph theory case where we get
some multiplicativity results using general facts about tensor products of operator spaces and operator systems.
Let us examine more closely the case when $\cl S=M_n$ and $\cl J=\cl S_G^{\perp}$. Throughout this section let $ X\in M_n$ and $Y\in M_m$.
 This means we can define the following two 
families of parameters,
 $$\sigma(G,X):=||X+\cl S_G^{\perp}||_{osy}$$  
 $$d_{\infty}(G,X):=||X+\cl S_G^{\perp}||_{osp}.$$
We will prove that given two graphs $G$ and $H$:
\[ \sigma(G \boxtimes H, X \otimes Y) = \sigma(G,X) \sigma(H,Y),\]
and \[ d_{\infty}(G \boxtimes H, X \otimes Y) = d_{\infty}(G, X)
d_{\infty}(H, Y),\]
for any matrices $X$ and $Y.$ 

In parallel with Lov\'{a}sz's
work, of special interest are the cases when these matrices are $I, A_G,$
and $R_G,$ which are all real symmetric matrices. 
Finally, for real matrices we give formulas for these quotient norms in
terms of SDP's which are then easy to implement and find numerically.

\begin{remark}
Our results can be extended to $||(X_{i,j}+\cl S_G^{\perp})||$, in either
the operator space or operator system case, by using the graph $G
\boxtimes K_m.$
\end{remark}

Before tackling our next result we need the following elementary lemma.

\begin{lemma}\label{perp} Let $G$ be a graph on $n$ vertices and let
  $H$ be a graph on $m$ vertices. Then
$$\cl S_G^{\perp} \otimes M_m + M_n \otimes \cl S_H^{\perp}=\cl S_{G \boxtimes H}^{\perp}$$
\end{lemma}
\begin{proof}
Let $X\otimes Y\in M_n\otimes \cl S_H^{\perp}$ and $N\otimes M\in \cl S_G\otimes \cl S_H$.
Notice that,
$$\inner{X\otimes Y}{N\otimes M}=\inner{X}{N}\inner{Y}{M}=0$$
This implies that,
$$M_n\otimes \cl S_H^{\perp}\perp \cl S_G\otimes \cl S_H$$
Similarly, 
$\cl S_G\otimes M_m\perp \cl S_G\otimes \cl S_H$.
Hence 
$$\cl S_G\otimes M_m+M_n\otimes \cl S_H^{\perp}\subseteq (\cl S_G\otimes
\cl S_H)^{\perp}.$$

Equality holds since they have the same dimensions. 

\end{proof}

\begin{thm}\label{ospmult} 
 Let $G$ be a graph on $n$ vertices with $X \in M_n$ and let
  $H$ be a graph on $m$ vertices with $Y \in M_m$. Then
 $$||X\otimes Y+\cl S_{G\boxtimes H}^{\perp}||_{osp}=
 ||X+\cl S_G^{\perp}||_{osp}\cdot ||Y+\cl S_H^{\perp}||_{osp},$$ that is,
 $d_{\infty}(G\boxtimes H, X \otimes Y) = d_{\infty}(G,X) \cdot
 d_{\infty}(H, Y).$
\end{thm}
\begin{proof}
Let $K\in \cl S_G^{\perp}$ and $L\in \cl S_H^{\perp}$ and notice the following,
\begin{align*}
||X+K||\cdot ||Y+L||&=|| (X+K)\otimes (Y+K)||\\
&=|| X\otimes Y + X\otimes L +K\otimes Y + K\otimes L ||
\end{align*}
Note that $X\otimes L +K\otimes Y + K\otimes L\in \cl S_G^{\perp} \otimes M_m + M_n \otimes \cl S_H^{\perp}$ and by \ref{perp} we have that $\cl S_G^{\perp} \otimes M_m + M_n \otimes \cl S_H^{\perp}=\cl S_{G \boxtimes H}^{\perp}$. Now if we take the infimum on both sides of the above equation, over all $K$ and $L$, we get,
$$||X+\cl S_G^{\perp}||_{osp}\cdot ||Y+\cl S_H^{\perp}||_{osp}\geq \inf\{ ||X\otimes Y +R|| : R\in \cl S_{G\boxtimes H}\}=||X\otimes Y+\cl S_{G\boxtimes H}^{\perp}||_{osp}.$$

The other inequality requires some results from the theory of
operator spaces. Let $Q_1: M_n \to M_n/\cl S_G^{\perp}$ and $Q_2:M_m \to M_m/ \cl S_H^{\perp}$ denote the quotient
maps. Since both of these maps are completely contractive by \cite[Thm. 12.3]{vpbook} the map
$Q_1 \otimes Q_2: M_n \otimes_{min} M_m \to (M_n/ \cl S_G^{\perp})
\otimes_{min} (M_m/ \cl S_H^{\perp})$ is completely contractive.
But $M_n \otimes_{min} M_m = M_{nm}$ and the kernel of $Q_1 \otimes
Q_2$ is $\cl S_{G \boxtimes H}^{\perp}.$  Hence,
\[ \|X \otimes Y + \cl S_{G \boxtimes H}^{\perp} \| \ge \|Q_1(X) \otimes
Q_2(Y)\| = \|Q_1(X)\| \cdot \|Q_2(Y)\|,\]
where the last equality follows from the fact \cite{bp} that the min tensor norm is a cross-norm.
We have that $\|Q_1(X)\|= \|X+ \cl S_G^{\perp}\|_{\osp}$ and $\|Q_2(Y)\|=
\|Y+ \cl S_H^{\perp}\|_{\osp}$ and so the proof is complete.
 \end{proof}


We now turn our attention to the operator space quotient norm in $M_n/\cl S_G^{\perp}$. Recall that  
$$||X+\cl S_G^{\perp}||_{osy}=\inf\Big\{\lambda : \begin{pmatrix} \lambda I+K_1 & X+K_2 \\ X^*+K^*_2 & \lambda I+K_3\end{pmatrix}\in M_2(M_n)^+,\ \mbox{for } K_i\in \cl S_G^{\perp}\Big\}.$$

\begin{thm}\label{subosy}Let $G$ and $H$ be graphs on $n$ and $m$ vertices, respectively, and let $X \in M_n$ and $Y \in M_m.$  Then
$$||X+\cl S_G^{\perp}||_{osy}||Y+\cl S_H^{\perp}||_{osy}= ||X\otimes Y+\cl S_{G\boxtimes H}^{\perp}||_{osy},$$ that is, $\sigma(G \boxtimes H, X\otimes Y) = \sigma(G, X) \cdot \sigma(H, Y).$
\end{thm}
\begin{proof}
We use the fact that \cite[Prop. 4.1]{kptt2010},
\begin{gather*}
||X +\cl S_G^{\perp}||_{osy}=sup\{||\phi_G(X)|| : \phi_G: M_n\rightarrow \cl B(\cl H),\ \phi_G(\cl S_G^{\perp})=0, \phi_G \ UCP \ \}\ (*)  
\end{gather*}
where the supremum is over all Hilbert spaces $H$ and UCP stands for ``unital, completely positive''.
Note that,
\begin{align*}
||X+\cl S_G^{\perp}||_{osy}||Y+\cl S_H^{\perp}||_{osy}&=\sup_{\phi_G, \phi_H}\{||\phi_G(X)||\cdot||\phi_H(Y)||\}\\
&=\sup_{\phi_G, \phi_H}\{||\phi_G(X)\otimes \phi_H(Y)||\} \\
&=\sup_{\phi_G, \phi_H}\{||\phi_G \otimes \phi_H(X\otimes Y)||\}
\end{align*}
where this supremum is over all maps that satisfy property $(*)$ and $\phi_G \otimes \phi_H(X\otimes Y)$ is the map that takes elementary tensors to the tensor of the corresponding images of the maps.
Notice $\phi_G \otimes \phi_H $ is a UCP map that vanishes on $\cl S_G^{\perp} \otimes M_m + M_n \otimes\cl S_H^{\perp}=\cl S_{G\boxtimes H}^{\perp}$ (\ref{perp}). Finally note,
\begin{gather*}
 \sup_{\phi_G, \phi_H} \{ \|\phi_G \otimes \phi_H(X \otimes Y)\|\} \leq ||X\otimes Y+\cl S_{G\boxtimes H}^{\perp}||_{osy}.
\end{gather*}

Thus,
\[ \| X \otimes Y + \cl S_{G \boxtimes H}^{\perp}\|_{osy} \ge \|X + \cl S_{G}^{\perp}\|_{osy} \|Y + \cl S_H^{\perp}\|_{osy}.\]

We now prove the other inequality:
$||X+\cl S_G^{\perp}||_{osy}||Y+\cl S_H^{\perp}||_{osy}\geq ||X\otimes Y+\cl S_{G\boxtimes H}^{\perp}||_{osy}.$

Let $\lambda > \|X+ \cl S_G^{\perp}\|_{\osy}$ and pick $K_i\in \cl S_G$ such that the $2n \times 2n$ block matrix
$$\begin{pmatrix} \lambda I_n+K_1 & X+K_2 \\ X^*+K^*_2 & \lambda I_n+K_3\end{pmatrix} \ge 0.$$
Similarly, let $\mu > \|Y + \cl S_H^{\perp}\|_{\osy}$ and pick
  $L_i\in \cl S_H$ such that the $2m \times 2m$ matrix
$$\begin{pmatrix} \mu I_m+L_1 & Y+L_2 \\ Y^*+L^*_2 & \mu I_m+L_3\end{pmatrix} \ge 0.$$

Tensoring these matrices we have that the $4mn \times 4mn$ block matrix,
\[
\begin{pmatrix} \lambda I_n+K_1 & X+K_2 \\ X^*+K^*_2 & \lambda I_n+K_3\end{pmatrix}\otimes\begin{pmatrix} \mu I_m+L_1 & Y+L_2 \\ Y^*+L^*_2 & \mu I_m+L_3\end{pmatrix} \ge 0 .\]


Restricting to the 4 blocks that occur in the corners we see that
$$\begin{pmatrix} (\lambda I_n+K_1)\otimes (\mu I_m+L_1) & (X+K_2)\otimes(Y+L_2) \\ (X^*+K^*_2)\otimes (Y^*+L^*_2) & (\lambda I_n+K_3)\otimes (\mu I_m+L_3) \end{pmatrix} \ge 0$$
But this matrix is of the form
$$\begin{pmatrix} \lambda\cdot\mu (I_n \otimes I_m)+Q_1 & X\otimes Y+Q_2 \\ (X\otimes Y +Q_2)^* & \lambda\cdot \mu(I_n \otimes I_m)+Q_3 \end{pmatrix}$$
for some $Q_i\in \cl S_G\otimes M_m+M_n\otimes \cl S_H=\cl S_{G\boxtimes H}^{\perp}$.
From this it follows that 
\[ \| X \otimes Y + \cl S_{G \boxtimes H}^{\perp}\|_{\osy} \le \lambda \mu.\]
Since $\lambda$ and $\mu$ were arbitrary,
$$||X+\cl S_G^{\perp}||_{osy}||Y+\cl S_H^{\perp}||_{osy}\geq ||X\otimes Y+\cl S_{G\boxtimes H}^{\perp}||_{osy}$$
and the proof is complete.
\end{proof}

For the purposes of numerical calculation it is often convenient to have
 dual formulations for computing 
 $||X+\cl S_G^{\perp}||_{osp}$ and $\|X + \cl S_G^{\perp}\|_{\osy},$
 especially in the case that $X$ is a real matrix. We write $M_n(\bb
 R)$ for the set of real matrices and $X^T$ for the transpose of the matrix $X.$

\begin{prop}\label{achieving} Let $G$ be a graph on $n$ vertices and
  let $X \in M_n(\bb R).$ Then
$||X+\cl S_G^{\perp}||_{osp}=\|X+H\|$ for some $H \in \cl S_G^{\perp}
\cap M_n(\bb R).$ 
\end{prop}
\begin{proof} Given a matrix $Y= (y_{i,j})$ we set $\overline{Y}= ( \overline{y}_{i,j}).$
Since $\cl S_G^{\perp}$ is a subspace of $M_n$ we know that there is a
$K\in \cl S_G^{\perp}$ such that $||X+\cl
S_G^{\perp}||_{osp}=||X+K||$. Now since $||X+K||=||\overline{X+K}||=
\|X + \overline{K}\|$ and $\overline{K}\in \cl S_G^{\perp}$ we get that,
$$||X+\cl S_G^{\perp}||_{\osp}\geq||X+\frac{K+\overline{K}}{2}||$$
so we have that $||X+\cl S_G^{\perp}||_{osp}=||X+H||$ where
$H=\frac{K+\overline{K}}{2}\in \cl S_G^{\perp} \cap M_n (\bb R).$
\end{proof}

\begin{prop}\label{dual} Let $G$ be a graph on $n$ vertices and let $X
  \in M_n(\bb R)$ be a real matrix.  Then
$$||X+\cl S_G^{\perp}||_{osp}=\max\{ Tr(X^T Q) : Q\in \cl S_G^{\perp}\cap M_n(\bb R), Tr(|Q|)\leq 1\}$$
\end{prop}
\begin{proof} This follows from general facts about the ``dual of a
  quotient'' in Banach space theory together with the fact that the
  trace norm is the dual of the operator norm. Alternatively, 
this is a consequence of Example $(34)$ in \cite{semidef}, which
states that for the following minimization problem,
$$||X+\cl S_G^{\perp}||_{osp}=\min\{||X+\sum_{i,j} k_{i,j} E_{ij} ||: E_{ij}\in \cl S_G^{\perp}, k_{ij}\in \mathbb{R} \}$$
its dual is given by,
\begin{equation*}
\begin{aligned}
& {\text{maximize}}
& & Tr(X^T Q) \\
& \text{subject to}
& & Tr( (E_{ij})^TQ)=0, \; E_{ij}\in \cl S_G^{\perp} \\ 
& & & Tr(|Q|)\leq 1,
\end{aligned}
\end{equation*}
where $E_{i,j}$ denote the usual matrix units.
Now since $Tr( (E_{ij})^TQ)=q_{ij}=0$, where $q_{ij}$ is the $ij$-entry of $Q$, we get our result.
\end{proof}

We now turn our attention to a dual formulation of
$||X+\cl S_G^{\perp}||_{osy}$ as an SDP, but just like in the case 
of the operator space norm we first need the following lemma,

\begin{lemma}
Let $G$ be a graph on $n$ vertices and
  let $X \in M_n(\bb R).$ Then the value of 
$||X+\cl S_G^{\perp}||_{osy}$ is achieved for some choice of $K_i \in \cl S_G^{\perp}
\cap M_n(\bb R)$, $i=1,2,3$ with $K_1=K_1^T, \, K_3=K_3^T.$ 
\end{lemma}
\begin{proof} Suppose $||X+\cl S_G^{\perp}||_{osy}=\lambda$. By definition, 
 $$\begin{pmatrix} \lambda I+K_1 & X+K_2 \\ X^*+K^*_2 & \lambda I+K_3\end{pmatrix}\geq 0$$ 
for some choice of $K_i\in \cl S_G^{\perp}$, $i=1,2,3$ with $K_1=K_1^*, \, K_3K_3^*.$ Now note that,
$$0\leq\overline{\begin{pmatrix} \lambda I+K_1 & X+K_2 \\ X^*+K^*_2 & \lambda I+K_3\end{pmatrix} }=\begin{pmatrix} \lambda I+\overline{K_1} & X+\overline{K_2} \\ X^*+\overline{K^*_2} & \lambda I+\overline{K_3}\end{pmatrix}$$
Finally if we average over this two positive matrices,
$$ \frac{1}{2}\Big[\begin{pmatrix} \lambda I+K_1 & X+K_2 \\ X^*+K^*_2 & \lambda I+K_3\end{pmatrix}+ \begin{pmatrix} \lambda I+\overline{K_1} & X+\overline{K_2} \\ X^*+\overline{K^*_2} & \lambda I+\overline{K_3}\end{pmatrix}]=\begin{pmatrix} \lambda I+\frac{K_1+\overline{K_1} }{2} & X+\frac{K_2+\overline{K_2} }{2} \\ X^*+\frac{K_2^*+\overline{K_2^*} }{2} & \lambda I+\frac{K_3+\overline{K_3} }{2}\end{pmatrix}$$
we get our desired result.
\end{proof}

\begin{prop}\label{sysdual} Let $G$ be a graph on $n$ vertices and let
  $X \in M_n(\bb R),$ then
$$||X+\cl S_G^{\perp}||_{osy}=\max\{ 2\cdot Tr(X^T B)
: \begin{pmatrix} A & B\\B^T & C \end{pmatrix}\in M_2(\cl S_G)^+ ,
Tr(A+C)= 1\},$$
with $A,B,C \in M_n(\bb R).$
\end{prop}
\begin{proof}
Notice that we can write $||X+\cl S_G^{\perp}||_{osy}$ as the following SDP:

minimize 
  \ \ $\inner{x}{c}$ \\
subject to
\[ \begin{pmatrix} \lambda I & X\\X^* & \lambda I \end{pmatrix} 
\\+ \sum_{(i,j)\in \overline{G}}\begin{pmatrix}k_{i,j}( E_{ij}+E_{ji})
  & z_{i,j}E_{i,j}\\z_{i,j}E_{j,i} & y_{i,j}(E_{i,j} + E_{j,i}) \end{pmatrix}\geq 0
\]
for $c =
    \begin{cases}
            1, &         \text{if } l=1\\
            0, &         \text{if } l\neq 1
    \end{cases}$
and $x=
	\begin{cases}
			\lambda, & \text{if } l=1 \\
			k_{ij},  & \text{if } 2\leq l\leq \lfloor \frac{dim(\cl S_G^{\perp})}{2}\rfloor\\
			y_{ij},  & \text{if } \lfloor \frac{dim(\cl S_G^{\perp})}{2}\rfloor < l \leq dim(\cl S_G^{\perp})\\
			z_{ij},  & \text{if } dim(\cl S_G^{\perp}) < l \leq \lfloor \frac{3 dim(\cl S_G^{\perp})}{2}\rfloor
	\end{cases}.$ 
	
Now by \cite{semidef} the dual of the above program is given by,
\begin{equation*}
\begin{aligned}
& {\text{maximize}}
& & 2\cdot Tr(X^T B) \\
& \text{subject to}
& & \begin{pmatrix} A & B\\B^T & C \end{pmatrix}\in M_2(\cl S_G)^+ \\ 
& & & Tr(A+C)=1.
\end{aligned}
\end{equation*}
Finally, we see that strong duality also holds for this SDP since we
can always pick,
$$x=\begin{cases} 
		\lambda=\max\limits_j \big{\{}\sum\limits_{i=1}^{n} |X_{ij}| \big{\}}+1, & \text{if } l=1\\
		0, & \text{if } l \neq 1
	\end{cases}
$$
($X_{ij}$ is the $ij$-entry of X) such that our constraint satisfies,
$$\begin{pmatrix} \lambda \cdot I & X\\X^* & \lambda \cdot I \end{pmatrix}>0.$$
\end{proof}
\begin{remark} The two multiplicativity theorems, Theorem~\ref{ospmult} and Theorem~\ref{subosy}, can be proven for real matrices $X$ and $Y$ using these two dual formulations.
\end{remark} 
\section{Quotient Norms as Graph Parameters}
Lov\'{a}sz's famous sandwich theorem says that
\[ \omega (G) \le \vartheta( \overline{G} ) \le \chi (G),\]
where $\omega(G)$ is the size of the largest clique in $G$ and
$\chi(G)$ is the chromatic number of $G.$
One of the many formulas for Lov\'{a}sz's theta function is that
\[ \vartheta(\overline{G}) = \min \{ \lambda_1(R_G +K): K=K^* \in \cl S_G^{\perp}
\},\]
where $\lambda_1$ denotes the largest eigenvalue.
 Note that by Proposition~\ref{achieving},
$$d_{\infty}(G, R_G)=\inf\{||R_G+K||: K=K^*=K^t\in \cl S_G^{\perp} \}.$$
Since for self-adjoint matrices their norm is the maximum of the
absolute values of their eigenvalues,
$$\vartheta(\overline{G}) \leq d_{\infty}(G, R_G).$$
The only potential difference between these two quantities is that for
any matrix $K=K^* \in \cl S_G^{\perp}$ with $\lambda_1(R_G+K)=
\vartheta(\overline{G})$ we have that
$-\lambda_n(R_G+K) > \lambda_1(R_G+K).$

This suggests we should examine the question of equality of these two
parameters and study the role that the potentially larger
$d_{\infty}(G, R_G)$ could play in sandwich type
theorems. 

We begin with an example where $\vartheta(\overline{G}) <
d_{\infty}(G, R_G).$
 For $G=C_6$ we know that $\vartheta(\overline{G})=2$, but $d_{\infty}(G, R_G)=2.25$. To see that this is the case notice that for any $K=K^*=K^t\in \cl S_G^{\perp}$
\begin{equation*}
A=\frac{\sum\limits_{k=0}^{5}{(S^{*})^k(R_G+K)S^k}}{6}=
\begin{pmatrix}
 1 & 1 & a & b & a & 1 \\
 1 & 1 & 1 & a & b & a \\
 a & 1 & 1 & 1 & a & b \\
 b & a & 1 & 1 & 1 & a \\
 a & b & a & 1 & 1 & 1 \\
 1 & a & b & a & 1 & 1 \\
\end{pmatrix}
\end{equation*}
where $S$ is the cyclic forward shift mod 6. Since $K$ is real and
symmetric, $a,b\in \mathbb{R}$ by \ref{achieving}. Now since
$||A||\leq ||R_G+K||$ for any $K\in \cl S_G^{\perp}$, we have that
$d_{\infty}(G, R_G)$ achieves its minimum value at such a matrix $A$
for some choice of $a$ and $b$. A similar argument shows that
$\lambda_1(R_G+K)$ achieves its minimum at such a matrix $A.$ Now notice that for this matrix we can explicitly compute its spectrum 
\begin{equation*}
\sigma(A)=\{-a-b+2, -a-b+2, 2 a-b-1, b-a,b-a, 2 a+b+3\}
\end{equation*}
and hence if we perform the following minimization we get that,
$$d_{\infty}(G, R_G)=\min_{a,b\in \mathbb{R}} \max\{|-a-b+2|, |2 a-b-1|, |b-a|, |2 a+b+3|\}=2.25$$
achieved when $a=-0.25$ and $b=0.5$. 

Similarly, minimizing $\lambda_1(A)$ over all $a$ and $b$ yields the
well-known fact that
$\vartheta(\overline{G}) =2.$

This fact gives rise to a new condition on the graph, namely, what happens when $\vartheta(\overline{G}) = d_{\infty}(G, R_G)$? 

Note that the orthogonal projection, $P_G: M_n \to \cl S_G$ is given
by Shur product with $R_G.$ Although $P_G$ has norm one when we regard
$M_n$ as a Hilbert space, in general, when we endow $M_n$ with the
usual operator norm then $\|P_G\|$ can be much larger than 1. It is
this latter norm that we are interested in.
For operator theorists, this is known as the {\it Schur multiplier
  norm of} $R_G$, sometimes denoted $\|R_G\|_m.$ For graph theorists,
this is sometimes denoted $\gamma(G).$ 

\begin{prop} If $\vartheta(\overline{G})=d_{\infty}(G, R_G)$, then
\[ \frac{1+ \lambda_1(A_G)}{\|P_G\|} \le \vartheta(\overline{G}).\]
\end{prop}
\begin{proof}
In \cite{lo} it was show that there exists a self-adjoint matrix of the form $R_G +K$ with $K \in \cl S_G^{\perp}$ such that
$\vartheta(\overline{G}) = \lambda_1(R_G +K)$. Now, if $\vartheta(\overline{G})=d_{\infty}(G, R_G)$, then $\|R_G +K\|= \lambda_1(R_G +K)$, and we get that
 
 \[ \|R_G\| = \|P_G(R_G+K)\| \le \|P_G\|\cdot \|R_G+K\|= \|P_G\|\vartheta(\overline{G}),\] 
so that
\[\frac{\|R_G\|}{\|P_G\|} \le \vartheta(\overline{G}).\] 

Also, it is the case that $\|A_G\|= \lambda_1(A_G)$\cite{sp}, from which it
follows that $\|R_G\|= \|I+A_G\|= 1 + \lambda_1(A_G).$
\end{proof}

\begin{cor}\label{newsandwich}If
  $\vartheta(\overline{G})=d_{\infty}(G, R_G)$,
  then \[ \frac{\chi(G)}{\|P_G\|} \le  \vartheta(\overline{G}) \le \chi(G).\]   
\end{cor}
\begin{proof} By Wilf's theorem \cite{sp},  $1 + \lambda_1(A_G) \ge \chi(G).$
\end{proof}


We now give at least one condition for when these parameters are
equal, although it is very restrictive. 

\begin{thm}
If $\vartheta(\overline{G})\leq 2$ then there exists a matrix $A$ satisfying $a_{ij}=1$ when i=j or $i\nsim j$ (1) with $\lambda_1(A)=\vartheta(G)=||A||$.
\end{thm}
\begin{proof}
By [5, Theorem 3] there exist a matrix $A$ satisfying (1) and $\vartheta(G)=\lambda_1(A)$ such that 
\begin{equation*}
\vartheta(G)I-A=(c-\sqrt{\vartheta(G)}\cdot u_i)^T(c-\sqrt{\vartheta(G)}\cdot u_j) \ \ \ \ \ (*)
\end{equation*}
with optimal orthonormal representation $(u_1,u_2,...,u_n)$ of $G$ with handle $c$ such that $\vartheta(G)=\frac{1}{(c^Tu_1)^2}=\cdots=\frac{1}{(c^Tu_n)^2}$. We must show that $-\lambda_n(A)\leq \vartheta(G)$. By $(*)$ we get that,
\begin{align*}
-A&=(c-\sqrt{\vartheta(G)}\cdot u_i)^T(c-\sqrt{\vartheta(G)}\cdot u_j)-\vartheta(G)I\\
&=c^Tc-\sqrt{\vartheta(G)}\cdot u_i^Tc-\sqrt{\vartheta(G)}\cdot c^Tu_j+\vartheta(G)\cdot u_i^Tu_j-\vartheta(G)I\\
&=1-1-1+\vartheta(G)\cdot u_i^Tu_j-\vartheta(G)I\\
&=\vartheta(G)\cdot u_i^Tu_j-\vartheta(G)I-J
\end{align*}
Now pick a unit vector $h$ such that $-\lambda_n(A)=\inner{-Ah}{h}$ and notice that,
\begin{align*}
-\lambda_n(A)&=\inner{-Ah}{h}\leq \vartheta(G)\inner{u_i^Tu_jh}{h}-\vartheta(G)\inner{h}{h}\\
&\leq\vartheta(G)||u_i^Tu_j||-\vartheta(G)=\vartheta(G)||I+H||-\vartheta(G).
\end{align*}
for some $H\in \cl S_G$. Now since $\vartheta(\overline{G})=\max \{||I+H|| : H\in \cl S_G\}$ we get, $$\vartheta(G)\vartheta(\overline{G})-\vartheta(G)\leq\vartheta(G).$$
\end{proof}

\begin{cor}
If $\vartheta(G)\leq 2$ then $\vartheta(\overline{G})=d_{\infty}(G, R_G)$.
\end{cor}

The condition $\vartheta(G) \le 2$ is quite restrictive. It is met by $K_n, C_4,
K_{2,\ldots, 2}$ and some graphs that are ``nearly'' complete.




\end{document}